\documentclass[12pt]{amsart}
\usepackage{amssymb, latexsym}
\usepackage{verbatim}

\newcommand{\akh}{\hat{a}_m}
\newcommand{\ak}{a_m}
\newcommand{\qh}{\hat{q}}
\newcommand{\q}{q}
\newcommand{\Flk}{\mathcal{F} ^\ell _k}
\newcommand{\floorn}{\lfloor \frac{n}{a_m}\rfloor}
\newcommand{\floornM}{\lfloor \frac{n}{M}\rfloor}

\theoremstyle{plain}
\newtheorem{theorem}{Theorem}[section]
\newtheorem{corollary}[theorem]{Corollary}
\newtheorem{lemma}[theorem]{Lemma}

\theoremstyle{remark}
\newtheorem{remark}{Remark}[section]
\newtheorem{example}[remark]{Example}

\begin{document}
\title[Compact symmetric solutions]
{Compact symmetric solutions to the postage stamp problem}
\thanks{Both authors were supported in part by the
National Sciences and Engineering Research Council of Canada.}
\author{Hugh Thomas}
\address{Department of Mathematics and Statistics, University of New
Brunswick, Fredericton NB, E3B 5A3 Canada}
\email{hugh@math.unb.ca}
\author{Stephanie van Willigenburg}
\address{Department of Mathematics, University 
of British Columbia, Vancouver BC,  V6T 1Z2 Canada}
\email{steph@math.ubc.ca}



\begin{abstract}
We derive lower and upper bounds on possible growth rates of
certain 
sets of positive integers $A_k=\{1= a_1 < a_2 < \ldots  < a_{k}\}$ such that all integers $n\in \{0, 1, 2, \ldots , ka_{k}\}$ can be represented as a sum of no more than $k$ elements of $A_k$, with repetition. 
\end{abstract}

\maketitle

\section{Introduction \label{intro}}
The postage stamp problem \cite[C 12]{guy}  is a classic problem in additive number theory and can be described as follows: if $h$ and $k$ are positive integers, $A_k=\{1= a_1 < a_2 < \ldots  < a_{k}\}$, $a_i\in \mathbb{N}$ and
$$S(h, A_k)=\{\sum _{i=1} ^{k} x_ia_i | x_i \geq 0, \sum _{i=1} ^{k} x_i \leq h\}$$
then what is the smallest positive integer $N(h,A_k)\not \in S(h, A_k)$? One focus is to solve the global aspect of this problem, that is, given $h$ and $k$ find $A_k$ such that $N(h, A_k)$ is as large as possible. The case $k=3$ was solved by Hofmeister \cite{hofmeister}, and for $k\geq 4$ R\"odseth \cite{rodseth} derived the currently best known general upper bound.  Another focus is to solve the local aspect, that is, given $h, k$ and $A_k$ determine $N(h, A_k)$. The case $k=3$ is covered in \cite{selmer}.  Both aspects were solved for the case $k=2$ in \cite {stohr}. 

It is easy to see that $N(h,A_k)\leq ha_k+1$.  In this paper, we focus on
integrating the global and local aspects by investigating  certain sets generated  for which this inequality is actually an equality.

\subsection{Preliminaries}
From here on we restrict our attention to the situation $h=k$. We say a set $A_k=\{1= a_1 < a_2 < \ldots  < a_{k}\}$ is \emph{symmetric} if when $k=2m$ then
\begin{eqnarray*}
a_1&=&1\\
a_i &>&a_{i-1} \text{ for } 2\leq i\leq m\\
a_{m+i}&=&2a_m-a_{m-i}\text{ for } 1\leq i \leq m-1\\
a_{2m}&=&2a_m
\end{eqnarray*}
and when $k=2m+1$ then
\begin{eqnarray*}\
a_1&=&1\\
a_i &>&a_{i-1} \text{ for } 2\leq i\leq m\\
\akh &=& a_m +x \text{ for } 0<x\in \mathbb{N} \\
a_{m+i}&=&\akh + a_m-a_{m-i}\text{ for } 1\leq i \leq m-1\\
a_{2m}&=&\akh + a_m,
\end{eqnarray*}where the $2m+1$ elements are ordered $$a_1<a_2<\ldots < a_m<\akh <a_{m+1}<\ldots <a_{2m}.$$This labelling of the elements has been chosen to make the enclosed proofs more uniform.

The largest integer that can be represented as the sum of $k$
positive integers chosen from $A_k$, with repetitions allowed, is clearly
$ka_k$.  If every positive integer $n$, $0\leq n\leq ka_k$, 
can be represented as the
sum of at most $k$ positive integers from $A_k$, then we say that $A_k$ is
\emph{compact}. We now study the growth rate of the $a_i$ such that $A_k$ is both symmetric and compact. More precisely, if $A_k=\{1= a_1 < a_2 < \ldots  < a_{2m}\}$ is symmetric then we derive bounds $\alpha ,\beta$ such that if $\frac{a_i}{a_{i-1}}\leq \alpha$ for $2\leq i\leq m$ then $A_k$ will always be compact, whereas if $\beta \leq \frac{a_i}{a_{i-1}}$ for $2\leq i\leq m$ then $A_k$ will never be compact. Symmetric compact sets were studied in \cite{doig} where the focus was on sets with a stronger symmetry property, known as nested symmetry.

For convenience, we refer to $A_k$ as the \emph{base}, refer to the $a_i$ as \emph{base elements}, and denote the largest base element by $M$.

\subsection{Acknowledgements.}  The authors are grateful to Greg Martin for suggesting the problem  and to GAP \cite{Gap} for generating the pertinent data. 

\section{A lower bound}
We now describe symmetric sets $A_k$ that are  compact.  For the remainder of this section, let $A_k=\{1= a_1 < a_2 < \ldots  < a_{2m}\}$ be a symmetric base such that 
\begin{enumerate}\item the $a_i$ satisfy
\begin{eqnarray*}
a_1&=&1\\
a_i &\leq &3a_{i-1} \text{ for } 2\leq i\leq m\\
a_{m+i}&=&2a_m-a_{m-i}\text{ for } 1\leq i \leq m-1\\
a_{2m}&=&2a_m
\end{eqnarray*}
or
\item the $a_i$ satisfy
\begin{eqnarray*}\
a_1&=&1\\
a_i &\leq& 3a_{i-1} \text{ for } 2\leq i\leq m\\
\akh &=& a_m +x \text{ for } 0< x \le 2a_m\\
a_{m+i}&=&\akh + a_m-a_{m-i}\text{ for } 1\leq i \leq m-1\\
a_{2m}&=&\akh + a_m.
\end{eqnarray*}\end{enumerate}

The following theorem on $A_k$ can be proved via \cite[Korollar]{kirfel}, however, we provide a direct proof, which begins with
\begin{lemma}\label{main_lemma}
Let $1\leq r\leq m-1$. If $\floornM \leq r$ and $n- \floornM M< a_{r+1}$ then $n$ can be written as a sum of at most $2r$ base elements with repetition.
\end{lemma}

\begin{proof}
We proceed by induction on $r$. When $r=1$ observe that $n-\floornM M < a_2\leq 3$, so $n=0,1,2,M,M+1,M+2$. That $n$ can be written as a sum of two base elements is trivial for all cases bar $n=M+2$. This case only arises if $a_2=3$,
in which case we can write $n=(M-1)+a_2$.

Now let $i=\floornM , j=\lfloor \frac{n-iM}{a_r}\rfloor$. Since $n-iM <a _{r+1} \leq 3 a_r$ we know $0\leq j \leq 2$. If
\begin{enumerate}
\item $j\leq 1$ let $n'= n- j a_r - \min(i,1)M$
\item $j=2, i=0$ let $n' = n-2a_r$
\item $j=2, i>0$ let $n'= n -(M-a_r)-a_{r+1}$.
\end{enumerate}

Note in each of these cases, respectively, $n'\geq 0$ since
\begin{enumerate}
\item if $j=0$ then $n-iM\geq 0$, whereas if $j=1$ then $n-iM \geq a_r$
\item if $j=2$ and $i=0$, then$\lfloor \frac{n}{a_r} \rfloor = 2$, so $n\geq 2a_r$
\item if $j=2$ and $i>0$, then $\lfloor \frac{n-iM}{a_r} \rfloor =2$, so $n-iM \geq 2a_r$ and $n-iM+a_r-a_{r+1}\geq 3a_r -a_{r+1}\geq 0$.
\end{enumerate}

Moreover, in each of these cases $n'$ respectively satisfies
\begin{enumerate}
\item $i'=\lfloor \frac{n'}{M} \rfloor = 0$ if $i=0$, or 
$i'=\lfloor \frac{n'}{M} \rfloor =i-1\leq r-1$ otherwise, and 
$\lfloor \frac{n'-i'M}{a_r} \rfloor=\lfloor \frac{n-ja_r}{a_r} \rfloor =0$ if $i=0$, or 
$\lfloor \frac{n'-i'M}{a_r} \rfloor =\lfloor \frac{n'-(i-1)M}{a_r} \rfloor=
\lfloor \frac{n-ja_r-iM}{a_r} \rfloor =0$ otherwise, so $n'- \lfloor \frac{n'}{M} \rfloor M < a_r$
\item $i'=\lfloor \frac{n'}{M} \rfloor = 0<r-1$ since $i=0$ and $r\geq 2$, and $\lfloor \frac{n'-i'M}{a_r} \rfloor = \lfloor \frac{n- 2a_r}{a_r} \rfloor =0$ since $j=2$, so  $n'- \lfloor \frac{n'}{M} \rfloor M < a_r$
\item $i'=\lfloor \frac{n'}{M} \rfloor =\lfloor \frac{n+ a_r - a_{r+1} -M}{M} \rfloor =i-1\leq r-1$ since $a_r-a_{r+1}<0$, and $\lfloor \frac{n'-i'M}{a_r} \rfloor =\lfloor \frac{n'-(i-1)M}{a_r} \rfloor = \lfloor \frac{n+ a_r - a_{r+1} -iM}{a_r} \rfloor =0$ since $n-iM< a_{r+1}$, so $n-iM-a_{r+1} +a_r < a_r$, and hence  $n'- \lfloor \frac{n'}{M} \rfloor M < a_r$.
\end{enumerate}

Thus in each case we can apply the induction hypothesis to $n'$ and write $n'$ as a sum of at most $2(r-1)$ base elements with repetition. The result now follows for $n$.
\end{proof}

\begin{theorem}\label{lower}
$A_k$ is compact.
\end{theorem}

\begin{proof} \emph{Case (1):} Suppose first that $n<2ma_m$.  
Consider $\floorn$. If $\floorn$ is even then $\floorn =2l\leq 2(m-1)$ and so $\lfloor \frac{n}{2a_m}\rfloor = \floornM \leq m-1$ and $n-\floornM M=n- \lfloor \frac{n}{2a_m}\rfloor M < a_m$. Thus, by Lemma ~\ref{main_lemma} $n$ can be written as a sum of at most $2(m-1)$ base elements. If $\floorn$ is odd then $\floorn = 2l+1$. Let 
$n' = n-a_m$ then $\lfloor \frac{n'}{a_m} \rfloor $ is even and by the above argument we can write $n'$ as a sum of at most $2(m-1)$ base elements, 
and hence in both situations $n$ can be written 
as a sum of at most $2m-1$ base elements. 

If $n=2ma_m$, it is clear that $n$ can be written as a sum of $m$ base elements,
so now suppose that $ 2ma_m< n \leq 4ma_m$.  
We have already shown that $4ma_m-n$ can
be written as a sum of at most $2m-1$ base elements; by the symmetry of $A_k$
it follows that $n$ can be written as a sum of at most $2m$ base elements.  

\emph{Case (2):} 
Suppose first that $n\leq (m+\frac{1}{2})(\akh +\ak)$.
Write
 $$n=\hat{q}\akh +q \ak +r$$
where $0\leq r< \ak$ and $|\qh -\q |$ is as small as possible. Note that in fact
 we can always find $\qh, \q$ such that $|\qh - \q |\leq 2$ by the following. If $|\qh - \q | >2$ then consider $\lfloor \akh / \ak\rfloor$. If $\lfloor \akh / \ak \rfloor =1$ and
 $\qh >\q $ then we can rewrite $n$ as 
 $$n= (\qh - 1)\akh + (\q +1) \ak  +(r+x)$$
 if $x+r <\ak$, and
 $$n= (\qh - 1)\akh + (\q +2) \ak +(r+x -\ak)$$
 otherwise. Alternatively if $\q >\qh$ then we can rewrite $n$ as
 $$n= (\qh + 1)\akh + (\q -2) \ak +(r-x +\ak)$$
 if $r<x$, or
 $$n= (\qh + 1)\akh + (\q -1) \ak +(r-x)$$
 otherwise. Similarly, if $\lfloor \akh / \ak \rfloor =2$ then we can rewrite $n$ as
 $$n= (\qh \mp 1)\akh + (\q \pm 2) \ak +(r \pm x \mp \ak)$$
 or
 $$n= (\qh \mp 1)\akh + (\q \pm 3) \ak +(r \pm x \mp 2\ak).$$
 Finally, if $\lfloor \akh / \ak \rfloor =3$, so $\hat a_m=3a_m$, then we can rewrite $n$ as
 $$n= (\qh \mp 1)\akh + (\q \pm 3) \ak +r.$$ 
 Iterating this procedure we see that we can eventually arrive at coefficients for $\akh$ and $\ak$ that differ by at most $2$. Thus let 
 $$n=\hat{q}\akh +q \ak +r$$
where $0\leq r<\ak$ and $|\qh -\q |\leq 2$.

Let  
\begin{eqnarray*}
n'&=&\left\{ \begin{array}{ll}
n-|\qh -q|\akh -\min (\q ,1)M&\mbox{ if $\qh >q$}\\
n-|\qh -q|\ak -\min (\qh ,1)M&\mbox{ if $q>\qh$}\\
n -\min (\qh ,1)M&\mbox{ if $\qh = q$.}
\end{array}\right .
\end{eqnarray*}Then $\lfloor \frac{n'}{\akh +\ak}\rfloor = \lfloor \frac{n'}{M}\rfloor \leq m-1$ and $n'- \lfloor \frac{n'}{M}\rfloor M < a_m$.
By Lemma ~\ref{main_lemma} we can write $n'$ as a sum of at most $2(m-1)$ base elements, and hence  $n$ can be written as a sum of at most $2m+1$ base elements. 

If $n> (m+\frac{1}{2})(\akh +\ak)$, we apply the symmetry of $A_k$ as in
Case (1), and we are done.  
\end{proof}

\section{An upper bound}
We now derive upper bounds on the $a_i$ by describing conditions on symmetric 
sets $A_k$ that force $A_k$ not to be compact. 
For the remainder of this section let 
 $A_k=\{1= a_1 < a_2 < \ldots  < a_{2m}\}$ be a symmetric base such that 
\begin{enumerate}\item the $a_i$ satisfy
\begin{eqnarray*}
a_1&=&1\\
a_i &\geq &8a_{i-1} \text{ for } 2\leq i\leq m\\
a_{m+i}&=&2a_m-a_{m-i}\text{ for } 1\leq i \leq m-1\\
a_{2m}&=&2a_m
\end{eqnarray*}
or
\item the $a_i$ satisfy
\begin{eqnarray*}\
a_1&=&1\\
a_i &\geq& 8a_{i-1} \text{ for } 2\leq i\leq m\\
\akh &=& a_m +x \text{ for } x \geq 7a_m\\
a_{m+i}&=&\akh + a_m-a_{m-i}\text{ for } 1\leq i \leq m-1\\
a_{2m}&=&\akh + a_m.
\end{eqnarray*}
\end{enumerate}

\begin{theorem}\label{upper}
$A_k$ is not compact.
\end{theorem}

\begin{proof}
\emph{Case (1):} Suppose that $A_k$ is compact.  If $n< a_m$, $n$
must be written as a sum of base elements chosen from $\{a_1,\dots,a_{m-1}\}$
using at most $2m$ summands.  There are $3m-1\choose m-1$ such sums.  Thus,
$a_m\leq {3m-1\choose m-1}$.  

Observe that 
$${3m\choose m-1}\leq {3m-1\choose i} \text{ for $m-1\leq i \leq 2m$.}$$
Thus, 
$${3m-1\choose m-1}<\frac{\sum_i {3m-1\choose i}}{m+2}=\frac
{2^{3m-1}}{m+2}\leq 2^{3m-3}$$
provided $m\geq 2$.  Thus, $a_m< 8^{m-1}$, contradicting our choice of 
$A_k$.

\emph{Case (2):} Similarly, suppose $A_k$ is compact.  If 
$n<a_m$, then $n$ must be written as a sum of base elements from
$\{a_1,\dots,a_{m-1}\}$, using at most $2m+1$ summands, so, by the same
argument as before, $a_m\leq {3m \choose m-1}$.  

Again, using the same argument, we find that
$${3m\choose m-1}< \frac{2^{3m}}{m+3}\leq 2^{3m-3},$$
provided $m\geq 5$, so $a_m< 8^{m-1}$, contradicting our choice
of $A_k$.  
\end{proof}

\begin{remark}
Utilising the above ideas and Stirling's formula it is possible to improve $8 \leq\frac{a_i}{a_{i-1}}$ to $\frac{27}{4} \leq\frac{a_i}{a_{i-1}}$ for $2\leq i\leq m$, however, we omit the lengthy calculations here. 
\end{remark}


\begin{thebibliography}{10}


\bibitem{Gap}
The GAP~Group, \emph{GAP -- Groups, Algorithms, and Programming,
Version 4.3}; 2002 \verb+(http://www.gap-system.org)+.
\bibitem{guy} R. Guy,  Unsolved problems in number theory. Third edition. Problem Books in Mathematics. Springer-Verlag, New York, 2004. 
\bibitem{hofmeister}G. Hofmeister,  \emph{Die dreielementigen Extremalbasen.} J. Reine Angew. Math.  339  (1983), 207--214.
\bibitem{kirfel} C. Kirfel, \emph{Stabilit\"at bei symmetrischen $h$-Basen.} Acta Arith. 51 (1988), no. 1, 85-96.
\bibitem{rodseth} \"O. R\"odseth,  \emph{An upper bound for the $h$-range of the postage stamp problem.}  Acta Arith.  54  (1990),  no. 4, 301--306.
\bibitem{selmer} E. Selmer,  \emph{On the postage stamp problem with three stamp denominations.}  Math. Scand.  47  (1980), no. 1, 29--71.
\bibitem{stohr} A.  St\"ohr, \emph{Gel\"oste und ungel\"oste Fragen \"uber Basen der nat\"urlichen Zahlenreihe. I, II.}  J. Reine Angew. Math.  194  (1955), 40--65, 111--140.
\bibitem{doig}P.  Wegner and A. Doig,  \emph{Symmetric solutions of the postage stamp problem.} Revue francaise de recherche operationelle 41 (1966), 353--374.
\end{thebibliography}
\end{document}
